\documentclass[12pt]{amsart}

\usepackage{amsfonts,amsmath,amsthm}
\usepackage{latexsym}

\newtheorem{theorem}{Theorem}[section]
\newtheorem{corollary}[theorem]{Corollary}

\newtheorem{remark}[theorem]{Remark}
\newtheorem{problem}[theorem]{Problem}

\newtheorem{conjecture}[theorem]{Conjecture}
\newtheorem{proposition}[theorem]{Proposition}

\theoremstyle{definition}

\newcommand{\dst}{\displaystyle}

\newcommand{\TT}{\ensuremath{\mathbb{T}}}

\newcommand{\ZZ}{\ensuremath{\mathbb{Z}}}

\newcommand{\CC}{\ensuremath{\mathbb{C}}}
\newcommand{\Co}{\ensuremath{\mathbb{C}}}

\def \C {\mathbb{C}}

\newcommand{\ac}{\ensuremath{\mathcal{A}}}
\newcommand{\bc}{\ensuremath{\mathcal{B}}}
\newcommand{\gc}{\ensuremath{\mathcal{G}}}
\newcommand{\ig}{\ensuremath{\mathcal{G}}}

\newcommand{\vb}{\ensuremath{\mathbf{v}}}
\newcommand{\eb}{\ensuremath{\mathbf{e}}}
\newcommand{\fb}{\ensuremath{\mathbf{f}}}
\newcommand{\ub}{\ensuremath{\mathbf{u}}}
\newcommand{\wb}{\ensuremath{\mathbf{w}}}

\newcommand{\cb}{\ensuremath{\mathbf{c}}}
\newcommand{\sbb}{\ensuremath{\mathbf{s}}}

\def \< {\langle}
\def \> {\rangle}

\newcommand{\ent}[1]{{\left[{#1}\right]}}
\newcommand{\abs}[1]{{\left|{#1}\right|}}
\newcommand{\scal}[1]{{\left\langle{#1}\right\rangle}}

\begin{document}

\title[Real and complex unbiased Hadamard matrices]{Real and complex unbiased Hadamard matrices}

\author[M. Matolcsi]{M. Matolcsi}
\address{M. M.: Alfr\'ed R\'enyi Institute of Mathematics,
Hungarian Academy of Sciences POB 127 H-1364 Budapest, Hungary
Tel: (+361) 483-8307, Fax: (+361) 483-8333}
\email{matomate@renyi.hu}

\author[I.Z. Ruzsa]{I.Z. Ruzsa}
\address{I.Z. R.: Alfr\'ed R\'enyi Institute of Mathematics,
Hungarian Academy of Sciences POB 127 H-1364 Budapest, Hungary
Tel: (+361) 483-8328, Fax: (+361) 483-8333}
\email{ruzsa@renyi.hu}

\author[M. Weiner]{M. Weiner}
\address{M. W.: Budapest University of Technology and Economics (BME),
H-1111, Egry J. u. 1, Budapest, Hungary
Tel: (+361) 463-2324}
\email{mweiner@renyi.hu}

\thanks{M.M supported by the ERC-AdG 228005, and OTKA Grants No. K81658, K77748, and the Bolyai Scholarship. I.Z. R. supported by ERC-AdG 228005, and OTKA Grant No. K81658. M. W. supported in part by the ERC-AdG 227458 OACFT}

\begin{abstract}
We use combinatorial and Fourier analytic arguments to prove various non-existence results on systems of real and complex unbiased Hadamard matrices. In particular, we prove that a complete system of complex mutually unbiased Hadamard matrices (MUHs) in any dimension $d$ cannot contain more than one real Hadamard matrix. We also give new proofs of several known structural results in low dimensions, for $d\le 6$.
\end{abstract}

\maketitle

\bigskip

\section{Introduction}

A new approach to the problem of mutually unbiased bases (MUBs) was recently given in \cite{mubfourier}, based on a general scheme in additive combinatorics. In  this paper we continue the investigations along this line, and prove several non-existence results concerning complete systems of MUBs, as well as some structural results in low dimensions. Let us remark here that the existence of MUBs is equivalent to the existence of mutually unbiased Hadamard matrices (MUHs) as explained below. In most of the paper it will be more convenient to deal with MUHs.

\medskip

The paper is organized as follows. The Introduction contains a standard summary of relevant notions and results concerning MUBs and MUHs. We also recall some elements of the general combinatorial scheme which was used in \cite{mubfourier}. In Section \ref{sec2} we use discrete Fourier analysis to prove several structural results on MUHs in low dimensions. Finally, in Section \ref{sec3} we prove non-existence results including the main result of the paper: a complete system of MUHs can contain at most one real Hadamard matrix. We also give a new proof, without using computer algebra, of the fact the Fourier matrix $F_6$ cannot be part of a complete system of MUHs in dimension 6. 

\medskip

Recall that two orthonormal bases in $\CC^d$,
$\ac=\{\eb_1,\ldots,\eb_d\}$ and $\bc=\{\fb_1,\ldots,\fb_d\}$ are
called \emph{unbiased} if for every $1\leq j,k\leq d$,
$\abs{\scal{\eb_j,\fb_k}}=\dst\frac{1}{\sqrt{d}}$. In general, we will say that two unit vectors $\ub$ and $\vb$ are \emph{unbiased} if $\abs{\scal{\ub,\vb}}=\dst\frac{1}{\sqrt{d}}$. A collection
$\bc_0,\ldots\bc_m$ of orthonormal bases is said to be
\emph{(pairwise) mutually unbiased} if every two of them are
unbiased. What is the maximal number of pairwise mutually unbiased bases (MUBs) in $\CC^d$? This question originates from quantum information theory and has been investigated thoroughly over the past decades (see \cite{durt} for a recent comprehensive survey on MUBs). The following result is well-known (see e.g. \cite{BBRV,BBELTZ,WF}):

\begin{theorem}\label{thm1}
The maximal number of mutually unbiased bases in $\Co^d$ is at most $d+1$.
\end{theorem}

Another important result concerns prime-power dimensions (see e.g.
\cite{BBRV,Iv,KR,WF}).

\begin{theorem}\label{thm2}
A collection of $d+1$ mutually unbiased bases (called a {\it complete set} of MUBs) exists
if the dimension $d$ is a prime or a prime-power.
\end{theorem}

However, if the dimension $d=p_1^{\alpha_1}\dots p_k^{\alpha_k}$ is composite then very little
is known except for the fact that there are at least $p_j^{\alpha_j}+1$ mutually
unbiased bases in $\C^d$ where $p_j^{\alpha_j}$ is the smallest of the prime-power divisors. In some specific square dimensions there is a construction based on orthogonal Latin squares which yields more MUBs than $p_j^{\alpha_j}+1$ (see \cite{262}). It is also known \cite{mweiner} that the maximal number of MUBs cannot be exactly $d$ (i.e. it is either $d+1$ or strictly less than $d$).

The following basic problem remains open for all non-primepower dimensions:

\begin{problem}\label{MUB6problem}\
Does a complete set of $d+1$ mutually unbiased bases exist in $\Co^d$ if $d$ is not a prime-power?
\end{problem}

The answer is not known even for $d=6$, despite considerable efforts over the past few years
(\cite{BBELTZ,config,ujbrit,arxiv}). The case $d=6$ is particularly tempting because it seems to be the simplest to handle with algebraic and numerical methods.  As of now, numerical evidence suggests that the maximal number of  MUBs for $d=6$ is 3 (see \cite{config,ujbrit,numerical,Za}).

\medskip

It will also be important for us to recall that mutually unbiased bases
are naturally related to mutually unbiased \emph{complex Hadamard matrices}.
Indeed, if the bases $\bc_0,\ldots,\bc_m$ are mutually
unbiased we may identify each
$\bc_l=\{\eb_1^{(l)},\ldots,\eb_d^{(l)}\}$ with the \emph{unitary}
matrix
$$
[U_l]_{j,k}=\ent{\scal{\eb_j^{(0)},\eb_k^{(l)}}_{1\leq k,j\leq
d}},
$$
{\it i.e.} the $k$-th column of $U_l$ consists of the
coordinates of the $k$-th vector of $\bc_l$ in the basis $\bc_0$.
(Throughout the paper the scalar product $\scal{.,.}$ of $\C^d$ is
conjugate-linear in the first variable and linear in the second.) With this convention,
$U_0=I$ the identity matrix, and all other matrices are unitary
and have all entries of modulus $1/\sqrt{d}$. Therefore, for $1\le l\le m$ the matrices
$H_l=\sqrt{d}U_l$ have all entries of modulus 1 and complex orthogonal
rows (and columns). Such matrices are called \emph{complex
Hadamard matrices}. It is thus clear that the existence of a family of $m+1$
mutually unbiased bases $\bc_0,\ldots,\bc_m$ is equivalent to
the existence of a family of $m$ complex Hadamard matrices
$H_1,\ldots, H_m$ such that for all
 $1\leq j\not=k\leq
m$, $\frac{1}{\sqrt{d}}H^{*}_jH_k$ is again a complex Hadamard matrix. In such
a case we will say that these complex Hadamard matrices are {\it
mutually unbiased} (MUHs). A system $H_1,\ldots, H_m$ of MUHs is called \emph{complete} if $m=d$ (cf. Theorem \ref{thm1}). We remark that there has been a recent interest in \emph{real} unbiased Hadamard matrices \cite{hadi1, hadi2, lecompte}, and the main result of this paper is that no pair of real unbiased Hadamard matrices can be part of a complete system of MUHs (see Corollary \ref{cornoreal}). The system $H_1, \dots H_m$ of MUHs will be called \emph{normalized} if the first column of $H_1$ has all coordinates 1, and all the columns in all the matrices have first coordinate 1. It is clear that this can be achieved by appropriate multiplication of the rows and columns by umimodular complex numbers. We will also use the standard definition that two complex Hadamard matrices $H_1$ and $H_2$ are equivalent, $H_1\cong H_2$, if
$H_1=D_1P_1H_2P_2D_2$ with unitary diagonal matrices $D_1, D_2$
and permutation matrices $P_1, P_2$.

\medskip

The crucial observation in \cite{mubfourier} is that the columns of $H_1, \dots , H_m$ can be regarded as elements of the group $\mathcal{G}=\TT^d$, where $\TT$ stands for the complex unit circle. By doing so, we can use Fourier analysis on $\gc$ to investigate the problem of MUHs.
We will now collect some notations that will be used in later sections (the notations in this paper are somewhat different and more convenient than in \cite{mubfourier}).
The group operation in $\gc$ is complex multiplication in each coordinate. The dual group is $\hat{\gc}=\ZZ^d$, and the action of a character $\gamma=(r_1,r_2, \dots , r_d)\in \ZZ^d$ on a group element $\vb=(v_1, v_2, \dots , v_d) \in \TT^d$ is given by exponentiation in each coordinate $\gamma(\vb)=\vb^\gamma=v_1^{r_1}v_2^{r_2}\dots v_d^{r_d}$. The Fourier transform of (the indicator function of) a set $S\subset \ig$ is given as $\hat{S}(\gamma)=\sum_{\sbb\in S} \sbb^\gamma$.

\medskip

As in \cite{mubfourier}, introduce the orthogonality set $ORT_d=\{\vb=(v_1, \dots ,v_d)\in \TT^d \ : \ v_1+ \dots + v_d=0\}$, and the unbiasedness set $UB_d=\{\vb=(v_1, \dots ,v_d)\in \TT^d \ : \ |v_1+ \dots + v_d|^2-d=0\}$. Then the (coordinate-wise) quotient $\vb/\ub$ of any two columns from the matrices $H_1, \dots H_m$ will fall into either $ORT_d$ (if $\vb$ and $\ub$ are in the same matrix) or into $UB_d$ (if  $\vb$ and $\ub$ are in different matrices). This enables one to invoke the general combinatorial scheme which we called "Delsarte's method": we refer the reader to \cite{mubfourier} for the details.

\section{Structural results on MUBs in low dimensions}\label{sec2}

In what follows we will assume that a \emph{complete} system of MUHs $H_1, \dots H_d$ is given. In fact, much of the discussion below remains valid for non-complete systems after appropriate modifications, but it will be technically easier to restrict ourselves to the complete case. The general aim is to establish structural properties of $H_1, \dots H_d$ which give restrictions on what a complete system may look like. If some of these properties were to contradict each other in a non-primepower dimension $d$, then we could conclude that a complete system of dimension $d$ does not exist. This is one of the main tasks for future research, mainly for $d=6$. We will give some non-existence results in this direction in Section \ref{sec3}.

\medskip

Consider each appearing complex Hadamard matrix $H_j$ as a $d$-element set in $\TT^d$ (the elements are the columns $\cb_1, \dots \cb_d$ of the matrix; the dependence on $j$ is suppressed for simplicity), and introduce its Fourier transform
\begin{equation}\label{gj}
g_j(\gamma):=\hat{H_j}(\gamma)= \sum_{k=1}^d \cb_k^\gamma \ \ \ \ \ \ \textrm{for each} \ \gamma\in \ZZ^d.
\end{equation}
Notice that the orthogonality of the \emph{rows} of $H_j$ implies that if $\rho\in\ZZ^d$ is any permutation of the vector $(1,-1,0,0,\dots ,0)$ then
\begin{equation}\label{gj0}
g_j(\rho)=0.
\end{equation}
Also, note that conjugation is the same as taking reciprocal for unimodular numbers, i.e. $\overline{g_j(\gamma)}=\sum_{k=1}^d \cb_k^{-\gamma}$, and therefore the square of the modulus of $g_j(\gamma)$ can be written as
\begin{equation}\label{Gj}
G_j(\gamma):=|g_j(\gamma)|^2= \sum_{k, l=1}^d (\cb_k/\cb_l)^\gamma \ \ \ \ \ \ \textrm{for each} \ \gamma\in \ZZ^d.
\end{equation} Also, introduce the notation
\begin{equation}\label{G}
G(\gamma):=\sum_{j=1}^d G_j(\gamma) \ \ \ \ \ \ \textrm{for each} \ \gamma\in \ZZ^d.
\end{equation}
In similar fashion, introduce the Fourier transform of the whole system as
\begin{equation}\label{f}
f(\gamma):=\sum_{j=1}^d g_j(\gamma) \ \ \ \ \ \ \textrm{for each} \ \gamma\in \ZZ^d, \ \ \ \textrm{and}
\end{equation}
\begin{equation}\label{F}
F(\gamma):=|f(\gamma)|^2 =\sum_{\ub,\vb}^d (\ub/\vb)^\gamma\ \ \ \ \ \ \textrm{for each} \ \gamma\in \ZZ^d,
\end{equation}
where the summation goes for all pair of columns $\ub, \vb$ in the matrices $H_1, \dots , H_d$.

\medskip

The main advantage of taking Fourier transforms is that any polynomial relation (such as orthogonality or unbiasedness) among the entries of the matrices $H_j$ will be turned into a \emph{linear} relation on the Fourier side. We will collect here linear equalities and inequalities concerning the functions $F(\gamma)$ and $G(\gamma)$.

\medskip

Let $\pi_r=(0,0,\dots 0,1,0, \dots 0)\in \ZZ^d$ denote the vector with the $r$th coordinate equal to 1. Then for each $j=1, \dots d$ we have
\begin{equation*}\label{Gjtileproof}
\sum_{r=1}^d G_j(\gamma+\pi_r)=\sum_{r=1}^d \left (\sum_{k, l=1}^d (\cb_k/\cb_l)^{\gamma+\pi_r} \right )= \sum_{k, l=1}^d (\cb_k/\cb_l)^{\gamma}\left (\sum_{r=1}^d (\cb_k/\cb_l)^{\pi_r} \right ),
\end{equation*}
and observe that the last sum is zero by orthogonality if $k\ne l$, while it is $d$ if $k=l$. This means that for each $j=1, \dots d,$
\begin{equation}\label{Gjtile}
\sum_{r=1}^d G_j(\gamma+\pi_r)=d^2 \ \ \ \ \ \ \textrm{for each} \ \gamma\in \ZZ^d,
\end{equation}
which then implies
\begin{equation}\label{Gtile}
\sum_{r=1}^d G(\gamma+\pi_r)=d^3 \ \ \ \ \ \ \textrm{for each} \ \gamma\in \ZZ^d.
\end{equation}

\medskip

In a similar fashion we can turn the unbiasedness relations also to linear constraints on the Fourier side. Let $\ub/\vb=(z_1, z_2\dots ,z_d)\in \TT^d$ be the coordinate-wise quotient of any two columns from two \emph{different} matrices from $H_1, \dots H_d$. Then $\ub$ and $\vb$ are unbiased, which means that
\begin{equation}\label{ubinz}
0=|\sum_r z_r|^2-d=\sum_{r\ne t}z_r/z_t.
\end{equation}
Using this we can write
\begin{equation}\label{FGtile}
\sum_{r\ne t} (F-G)(\gamma+\pi_r-\pi_t)= \sum_{\ub, \vb} (\ub/\vb)^{\gamma}\left (\sum_{r\ne t} (\ub/\vb)^{\pi_r-\pi_t} \right )=0,
\end{equation}
where the summation on $\ub, \vb$ goes for all pairs of columns from \emph{different} matrices, and the last equality is satisfied because each inner sum is zero by \eqref{ubinz}. Also, by \eqref{Gtile} we have $dG(\gamma)+\sum_{r\ne t} G(\gamma+\pi_r-\pi_t)=d^4$, and we can use this to rewrite \eqref{FGtile} as
\begin{equation}\label{FGtile2}
dG(\gamma)+\sum_{r\ne t} F(\gamma+\pi_r-\pi_t)=d^4,
\end{equation}
which is somewhat more convenient than \eqref{FGtile}.

\medskip

We also have some further trivial constraints on $F$ and $G$. Namely,
\begin{equation}\label{F0G0}
F(0)=d^4, \ \ \ G(0)=d^3, \ \ \ \ \ \textrm{and}
\end{equation}
\begin{equation}\label{FGpositive}
0\le F(\gamma)\le d^4, \ \ \ 0\le G(\gamma)\le d^3, \ \ \ \textrm{for each} \ \gamma\in \ZZ^d.
\end{equation}
Also, by the inequality of the arithmetic and quadratic means we have
\begin{equation}\label{FleG}
F(\gamma)\le d G(\gamma), \ \ \ \ \ \ \ \ \ \ \ \ \textrm{for each} \ \gamma\in \ZZ^d.
\end{equation}

\medskip

The point is that the linear constraints \eqref{Gtile}, \eqref{FGtile2}, \eqref{F0G0}, \eqref{FGpositive}, \eqref{FleG} put severe restrictions on the functions $F$ and $G$. In fact, it turns out that \emph{all} the structural results on complete systems of MUHs in dimensions 2, 3, 4, 5 follow from these constraints. These structural results are not new (cf. \cite{BWB}) but nevertheless we list here the two most important ones as an illustration of the power of this Fourier approach. The first one is a celebrated theorem of Haagerup \cite{haagerup} which gives a full classification of complex Hadamard matrices of order 5. In the original paper \cite{haagerup} the author combines several clever ideas with lengthy calculations to derive the result, whereas it follows almost for free from the formalism above.

\begin{proposition}\label{haa}
Any complex Hadamard matrix of order 5 is equivalent to the Fourier matrix $F_5$, given by $F_5(j,k)=\omega^{(j-1)(k-1)}$, ($j,k=1,\dots, 5$), where $\omega=e^{2i\pi/5}$.
\end{proposition}
\begin{proof}
Let $H_1$ be a complex Hadamard matrix of order 5. Then the function $G_1(\gamma)=|\hat{H_1}(\gamma)|^2$ satisfies equation \eqref{Gjtile} for all $\gamma\in \ZZ^5$, and we have $G_1(0)=25$ and $0\le G_1(\gamma)\le 25$ for all $\gamma\in \ZZ^5$. Regarding each $G_1(\gamma)$ as a nonnegative variable (as $\gamma$ ranges through a sufficiently large cube around the origin in $\ZZ^5$), a short linear programming code testifies that under these conditions $G_1(\rho)=25$ for all such $\rho\in \ZZ^5$ which is a permutation of $(5,-5,0,0,0)$. Also, we may assume without loss of generality that $H_1$ is normalized (i.e. its first row and column are made up of 1s), and then the information $G_1(\rho)=25$ implies that all other entries of $H_1$ are 5th roots of unity. It is then trivial to check that there is only one way (up to equivalence) to build up a complex Hadamard matrix from 5th roots of unity, namely the matrix $F_5$.
\end{proof}

We remark here all the linear programming mentioned in this paper uses rational coefficients, so no numerical errors are encountered, and each result is certifiable.
Let us also remark that Proposition \ref{haa} is the only \emph{non-trivial} result concerning MUHs and MUBs in dimensions $d\le 5$. The classification of complex Hamamard matrices and MUBs is more or less trivial for $d=2, 3, 4$ due to the geometry of complex unit vectors. We give here the essence of this classification (for full details see \cite{BWB}).

\begin{proposition}\label{2345}
In any normalized complete system of MUHs in dimension $d=3, 4, 5$ all entries of the matrices are $d$th roots of unity. For $d=2$ all entries are 4th roots of unity.
\end{proposition}
\begin{proof}
Let $d=3, 4, 5$. Assume $H_1, \dots H_d$ is a normalized complete system of MUHs. Then the functions $F$ and $G$ must satisfy the linear constraints \eqref{Gtile}, \eqref{FGtile2}, \eqref{F0G0}, \eqref{FGpositive}, \eqref{FleG}. Regarding each $F(\gamma)$ and $G(\gamma)$ as a nonnegative variable (as $\gamma$ ranges through a sufficiently large cube around the origin in $\ZZ^d$), a short linear programming code testifies that under these conditions $F(\rho)=d^4$ for all such $\rho\in \ZZ^d$ which is a permutation of $(d,-d,0,\dots ,0)$. This means that all entries in all of the matrices must be $d$th roots of unity. The proof is analogous for $d=2$ except that in this case we can only conclude $F(4,-4)=16$, so that the matrices contain 4th roots of unity.
\end{proof}

\medskip

Let us make a remark here about $d=4$. In this case it is \emph{not true} that all normalized Hadamard matrices must be composed of 4th roots of unity. However, it is true that a complete system of MUHs must be composed of such. This phenomenon shows up very clearly in our linear programming codes. Writing the constraints \eqref{Gjtile} on $G_1(\gamma)$, and $G_1(0)=16$, and $0\le G_1(\gamma)\le 16$ does not enable us to conclude that $G_1(\rho)=16$ with $\rho$ being a permutation of $(4,-4,0,0$. However, writing all the constraints \eqref{Gtile}, \eqref{FGtile2}, \eqref{F0G0}, \eqref{FGpositive}, \eqref{FleG} on the functions $F$ and $G$ we can indeed conclude that $F(\rho)=4G(\rho)=256$.

\medskip

We end this section with a few remarks concerning $d=6$. If we could similarly conclude that
\begin{equation}\label{d6}
F(\rho)=6^4 \ \  \ \ \ \textrm{for all} \ \rho \ \textrm{being a permutation of} \ (6,-6,0,0,0,0)
\end{equation}
then it would mean that a complete system of normalized MUHs in dimension 6 can only be composed of 6th roots of unity. Such a structural information would be wonderful, as it is easy to check by computer that no such complete system of MUHs exists. Therefore, we could conclude that a complete system of MUHs does not exist at all. Unfortunately, the constraints \eqref{Gtile}, \eqref{FGtile2}, \eqref{F0G0}, \eqref{FGpositive}, \eqref{FleG} do not seem to imply \eqref{d6}. At least, we have run a linear programming code with $\gamma$ ranging through as large a cube as possible (due to computational limitations), and could not conclude \eqref{d6}. Nevertheless, our main strategy for future research in dimension 6 must be as follows: using the linear constraints on $F$ and $G$ try to establish some structural information on the vectors appearing in a hypothetical complete system of MUHs, and then show by other means (e.g. a brute force computer search) that such constraints cannot be satisfied. We formulate here one conjecture which could be crucial in proving the non-existence of a complete system of MUHs in dimension 6.

\begin{conjecture}\label{conj}
Let $H_1$ be any complex Hadamard matrix of order 6, not equivalent to the isolated matrix $S_6$ (cf. \cite{karol} for the matrix $S_6$). Let $\rho$ be any permutation of the vector $(1,1,1,-1,-1,-1)$. Then $g_1(\rho)=0$ for the function $g_1$ defined in \eqref{gj}.
\end{conjecture}

This conjecture is supported heavily by numerical data. We have tried hundreds of matrices randomly from each known family of complex Hadamard matrices of order 6 (a complete classification is unfortunately not available). Currently we cannot prove this conjecture, but in Section 3 we will show an example of how it could be used in the proof of non-existence results (cf. Remark \ref{noFab}).

\section{Non-existence results}\label{sec3}

We now turn to non-existence results, namely that complete systems of MUHs with certain properties do not exist. The first of these, which we regard as the main result of the paper, is that any pair of \emph{real} unbiased Hadamard matrices cannot be part of a complete system of MUHs. In fact, we prove the following stronger statement.

\begin{theorem}\label{thmnoreal}
Let $H_1, \dots H_d$ be a complete system of MUHs such that $H_1$ is a real Hadamard matrix. Then any column vector $\vb=(v_1, \dots , v_d)$ of the other matrices $H_2, \dots H_d$ satisfies that $\sum_{k=1}^d v_k^2=0$.
\end{theorem}
\begin{proof}
Let $0\ne \rho=(r_1,\dots , r_d)\in\ZZ^d$ be such that $\sum_{k=1}^d r_k=0$ and $\sum_{k=1}^d |r_k|\le 4$. There are five types of these vectors (up to permutation): $(1,-1, 0,\dots ,0)$, $(2,-2,0,\dots ,0)$, $(2, -1, -1, 0, \dots, 0)$, $(-2, 1, 1, 0, \dots ,0)$, and $(1,1,-1,-1,0,\dots 0)$. Then, Theorem 8 in \cite{belovs} (or Corollary 2.4 in \cite{mubfourier}) shows that the function $f$ defined in \eqref{f} satisfies
\begin{equation}\label{Frho}
f(\rho)=0
\end{equation}
for all these vectors $\rho$.

\medskip

Let $\cb_1,\cb_2, \dots , \cb_{d^2}$ denote the column vectors appearing in the system $H_1, \dots H_d$. For each $\gamma\in \ZZ^d$ let
\begin{equation}\label{vgamma}
\vb(\gamma)=(\cb_1^\gamma, \dots \cb_{d^2}^\gamma)\in \TT^{d^2}
\end{equation}
for $k=1, \dots d$. Consider the vectors $\gamma_k=(0, \dots 0, 2, 0,\dots 0)\in \ZZ^d$ with the 2 appearing in position $k$. Finally, consider the vector $\wb = \sum_{k=1}^d \vb(\gamma_k)$, and let us evaluate $\|\wb\|^2$. On the one hand, the vectors $\vb(\gamma_k)$ are all orthogonal to each other by \eqref{Frho}, and they all have length $\|\vb(\gamma_k)\|^2=d^2$, and hence  $\|\wb\|^2=d^3$. On the other hand \emph{we know} the first $d$ coordinates of $\wb$. Each $\vb(\gamma_k)$ has first $d$ coordinates equal to 1, because $H_1$ is a real Hadamard matrix. Therefore the first $d$ coordinates of $\wb$ are all equal to $d$. Therefore, $\|\wb\|^2\ge d^3$ on account of the first $d$ coordinates. Hence, all other coordinates of $\wb$ must be zero, which is exactly the statement of the theorem.
\end{proof}

\medskip

Theorem \ref{thmnoreal} implies immediately the following corollary.
\begin{corollary}\label{cornoreal}
Let $H_1, \dots H_d$ be a complete system of MUHs such that $H_1$ is a real Hadamard matrix. Then there is no further purely real column in any of the matrices $H_2, \dots,H_d$. In particular, it is impossible to have two real Hadamard matrices in a complete set of MUHs.
\end{corollary}

This statement is sharp in the sense that for $d=2, 4$ the complete systems of MUHs are known to contain \emph{one} real Hadamard matrix. Also, in several dimensions $d=4n^2$ pairs (and even larger systems) of real unbiased Hadamard matrices are known to exist \cite{hadi1, hadi2}, so that the corollary above is meaningful and non-trivial. 

\medskip

Our next result is a new proof of the fact in dimension 6 the Fourier matrix $F_6$ cannot be part of a complete system of MUHs. This result is well-known, but the only proof we are aware of uses some computer algebra, while we present an easy conceptual proof here.

\begin{proposition}\label{noF6}
There exists no complete system of MUHs in dimension 6 which contains the Fourier matrix $F_6$.
\end{proposition}
\begin{proof}
Assume by contradiction that such a system $H_1, \dots H_6$ exists, and assume $H_1=F_6$.
Consider the vectors $\gamma_1=(1,1,1,0,0,1)$, $\gamma_2=(0,0,1,1,1,1)$, $\gamma_3=(1,1,0,1,1,0)$, $\gamma_4=(0,1,0,1,0,2)$, $\gamma_5=(1,0,0,0,1,2)$, and $\gamma_{6}=(0,1,0,0,2,1)$, and consider the corresponding vectors $\vb(\gamma_k)$ defined in \eqref{vgamma}, and let $\wb = \sum_{k=1}^6 \vb(\gamma_k)$. All the vectors $\vb(\gamma_k)$ are orthogonal to each other by \eqref{Frho}, therefore $\|\wb\|^2=216$. On the other hand, we know the first 6 coordinates of $\wb$. It is easy to calculate that each of these coordinates has modulus 6, and therefore $\|\wb\|^2\ge 216$ on account of the first 6 coordinates. This implies that all the other coordinates of $\wb$ must be zero. This yields a polynomial identity for the coordinates of any column vector appearing in the matrices $H_2, \dots , H_6$. Instead of using this identity directly, however, we observe that the same argument applies to the vectors $\gamma_1, \dots \gamma_5$ and $\gamma_{6}'=(2,0,0,1,0,1)$, and $\wb'=\vb(\gamma_{6}')+\sum_{k=1}^5 \vb(\gamma_k)$. By considering the difference $\wb-\wb'$ we conclude that $\vb(\gamma_6)$ and $\vb(\gamma_{6}')$ must coincide in the last 30 coordinates. That is, if $(z_1, \dots ,z_6)$ is any column vector in the matrices $H_2, \dots H_6$ then $z_2z_5^2z_6=z_1^2z_4z_6$, and hence $z_2z_5^2=z_1^2z_4$. Furthermore, one can  permute the coordinates of $\gamma_k$ in a cyclic manner, and the argument remains unchanged, yielding this time $z_5z_2^2=z_4^2z_1$. Dividing these two equations finally gives $z_5/z_2=z_1/z_4$ for each of the last 30 vectors in our complete system of MUHs. This means, by definition,  that the last 30 coordinates of the vectors $\vb(0,-1,0,0,1,0)$ and $\vb(1,0,0,-1,0,0)$ coincide. But this is a contradiction, because these vectors should be orthogonal to each other by \eqref{Frho}.
\end{proof}

\medskip

\begin{remark}\label{noFab}\rm
Finally, we discuss informally a non-existence result in which we use Conjecture \ref{conj} in the proof. Nevertheless, the result itself is not "conditional" because it was proved earlier in \cite{arxiv} by a massive computer search after a discretization scheme. The argument we present here is much more elegant though, and shows a possible way forward in proving the non-existence of complete systems of MUHs in dimension 6.

We claim that there exists no complete system $H_1,\dots ,H_6$ of MUHs in dimension 6 which contains any of the matrices $F_6(a,b)$ of the Fourier family (cf. \cite{karol} for the Fourier family $F_6(a,b)$). We sketch the proof here, on the condition that Conjecture \ref{conj} is valid. First of all, the statement is equivalent for the transposed family $F^T_6(a,b)$, as it is well-known that a matrix $H$ can be part of a complete system of MUHs if and only if $H^T$ can. The significance of this is that each member of the transposed family $F^T_6(a,b)$ contains the three vectors $\cb_1=(1,1,1,1,1,1)$, $\cb_2=(1,\omega,\omega^2,1,\omega,\omega^2)$ and $\cb_3=(1,\omega^2,\omega,1,\omega^2,\omega)$, where $\omega=e^{2i\pi/3}$. Also, it is well-known that a complex Hadamard matrix equivalent to $S_6$ cannot be part of a complete system of MUHs, so that we can assume without loss of generality that $H_1,\dots H_6$ are not equivalent to $S_6$. The significance of this fact is that now Conjecture \ref{conj} (if true) can be invoked. One can make a clever selection of vectors $\gamma_1, \dots \gamma_{12}\in \ZZ^6$ such that the same argument as in Proposition \ref{noF6} can be used. Namely, all the vectors $\vb(\gamma_k)$ are orthogonal to each other by either \eqref{Frho} or by Conjecture \ref{conj}, so that $\|\wb \|^2=|\sum_{k=1}^{12}\vb(\gamma_k)|^2=432$. (This is where we use Conjecture \ref{conj}; otherwise the appropriate selection of $\gamma_1,\dots \gamma_{12}$ would not be possible.) On the other hand, three coordinates of $\wb$ corresponding to the columns $\cb_1, \cb_2, \cb_3$ are known exactly, and the modulus of them happens to be 12, again by the choice of the vectors $\gamma_k$. This leads us to conclude that all the other 33 coordinates of $\wb$ must be zero, which yields a polynomial identity for the coordinates of each the 33 unknown columns in $H_1, \dots ,H_6$.  Finally, one can make several such selection of $\gamma_1, \dots \gamma_{12}$, each yielding a polynomial constraint on the unknown columns, and these constraints contradict each other just as in the proof of Proposition \ref{noF6}. $\hfill \square$
\end{remark}

\medskip

We believe that the proof of the non-existence of complete systems of MUHs in dimension 6 will hinge on Conjecture \ref{conj}. The reason is that it introduces yet another non-trivial linear constraint on the function $G$, and these constraints will ultimately lead to a contradiction (maybe indirectly, as in Proposition \ref{noF6}). Therefore, we would be very interested to see a proof of Conjecture \ref{conj}.

\end{document}